\documentclass[12pt]{amsart}
\usepackage{amscd}
\usepackage{amsmath}
\usepackage{hyperref}
\usepackage{hyperref}
\usepackage{tikz-cd}

\usepackage{amssymb}

\numberwithin{equation}{section}

\newtheorem{thm}[equation]{Theorem}
\newtheorem{prop}[equation]{Proposition}

\newtheorem{lemma}[equation]{Lemma}

\theoremstyle{definition}

\newtheorem{example}[equation]{Example}

\newtheorem{notation}[equation]{Notation}

\newcommand{\Gal}{\mathop{\mathrm{Gal}}}

\newcommand{\Th}{\mathop{\mathrm{Th}}}

\newcommand{\id}{\mathrm{id}}

\newcommand{\Z}{\mathbb{Z}}

\newcommand{\PP}{\mathbb{P}}
\newcommand{\Q}{\mathbb{Q}}

\newcommand{\R}{\mathbb{R}}

\newcommand{\Spec}{\operatorname{Spec}}

\newcommand{\Hom}{\operatorname{Hom}}

\newcommand{\Sym}{\operatorname{Sym}}

\renewcommand{\phi}{\varphi}

\newcommand{\WR}{\mathcal{R}}

\DeclareMathAlphabet{\cat}{OT1}{cmss}{m}{sl}

\title
[Weil Restriction, normal bundles and Motivic Thom spaces]
{Weil Restriction, normal bundles \\and Motivic Thom spaces }


\author
{Guangzhao Zhu}

\address
{Mathematical \& Statistical Sciences \\
University of Alberta \\
Edmonton
\\
CANADA}

\email
{guangzha@ualberta.ca}

\date
{\today}

\begin{document}

\begin{abstract}
Recent developments in motivic homotopy theory, particularly the construction of norm functors by Bachmann and Hoyois, have revealed deep connections between algebraic geometry and homotopy-theoretic structures. In this paper, we investigate certain geometric aspects of norm functors through the Weil restriction of schemes, which underlies these constructions. 

We show that Weil restriction preserves vector bundles and extend existing results concerning normal bundles. We then relate the Weil restriction to norm functors and, using a result of Bachmann and Hoyois, establish its compatibility with motivic Thom spaces. Finally, in the setting of motivic cohomology with rational coefficients, we prove that the Weil restriction map agrees with the norm map induced by a norm functor and that it preserves Thom classes.
\end{abstract}

\maketitle

\tableofcontents

\addtocounter{section}{0}

\section{Introduction}
The study of the motivic homotopy category originates from the pioneering work of Voevodsky and Morel \cite{MR1813224}, who introduced and systematically developed the notion of $\mathbb{A}^1$-homotopy for schemes. This framework allows one to study algebraic geometry using methods analogous to those of classical homotopy theory in topology. In particular, it provides a bridge between algebraic geometry and homotopical algebra, leading to powerful tools such as motivic cohomology, which plays a central role in the proof of the Bloch--Kato conjecture.

Analogous to the construction of the classical stable homotopy category in topology, a formal stabilization procedure—achieved by inverting the suspension with respect to the \emph{Tate object}—gives rise to the \emph{motivic stable homotopy category}. Within this stabilized category, diverse cohomology theories are naturally represented by motivic spectra, thereby providing a unified framework for studying generalized cohomology in algebraic geometry.

A prominent recent development in this area is the construction of norm functors on the motivic stable homotopy category for finite étale morphisms by Bachmann and Hoyois \cite{MR4288071}. These functors are symmetric monoidal and hence behave well with respect to multiplicative structures, in particular when passing to cohomology theories. 

Of particular importance is that they provide a conceptual explanation for various descent phenomena in motivic cohomology theories; see, for example, \cite{MR1809664} and \cite{WMC}.

While their comprehensive construction relies heavily on the machinery of $\infty$-categories, the underlying geometric operation behind these norm functors is the classical Weil restriction of schemes.

In this paper, we adopt a purely geometric approach to study the Weil restriction of vector bundles and normal bundles, and relate these constructions to motivic Thom spaces. This perspective highlights the geometric content underlying norm-type constructions without invoking highly abstract categorical machinery. Our approach is based on the Weil restriction map of motivic cohomology groups constructed by the author in \cite{WMC}.

We organize the paper as follows. In \S2, we recall basic properties of the Weil restriction of schemes and its relation to the norm functors constructed in the motivic stable homotopy category. In \S3, we study the Weil restriction of vector bundles and extend existing results on the Weil restriction of normal bundles. In \S4, we relate the Weil restriction of schemes to norm functors at the level of the unstable motivic homotopy category and Voevodsky's category of motives, and prove its compatibility with Thom spaces. Finally, in the setting of motivic cohomology with rational coefficients, we show that the Weil restriction map coincides with the norm map induced by norm functors and prove that it sends Thom classes to Thom classes.

\medskip
\noindent{\bf Notation and conventions.}

Throughout this paper, schemes are assumed to be separated and Noetherian. 

A \emph{variety} means an integral scheme, i.e., an irreducible and reduced scheme.

For any scheme $S$, we write $\textbf{Sm}_{S}$ for the category of smooth schemes over $S$.

For a finite étale morphism $f:S \to T$, we denote by $N_{f}$ the associated norm functor between the relevant motivic categories. 

We also write $N_{f}^{H}$ for the induced norm map on motivic cohomology groups.

\medskip
\noindent{\bf Acknowledgements.}
The author is grateful to Prof.~Xi Chen for helpful discussions and suggestions related to this project. The author is also grateful to Prof.~Stefan Gille for introducing the question.

\section{Weil restriction and norm functors}
\subsection{Basic properties of Weil restriction}
We briefly recall the definition and basic properties of the {\em Weil restriction} of schemes. For a comprehensive treatment, we refer the reader to \cite{MR1045822} or \cite{MR4480537}.
Let $f:S\to T$ be a morphism of schemes. For a scheme $X$ over $S$, the {\em Weil restriction} of $X$ along $f$, denoted by $\WR_f(X)$ (or $\WR_{S/T}(X)$ depending on the context, or simply $\WR(X)$ when $f$ is clear from the context), is a scheme (if exists) over $T$ characterized by the following universal property:
\[
\Hom_T(Y,\WR_f(X)) \cong \Hom_S(Y\times_T S, X)
\]
for every $T$-scheme $Y$.

For certain schemes, the existence of the Weil restriction is ensured by the {\em representability criterion} (See \cite[Theorem 4, \S7.6.~1]{MR1045822}). In particular, the Weil restriction exists whenever $X$ is quasi-projective over $S$. Moreover, the Weil restriction of a projective scheme remains projective; the Weil restriction of a quasi-projective scheme remains quasi-projective \cite[Proposition 2.10]{MR4480537}.

\begin{notation}
Unless otherwise specified, when we speak of the Weil restriction of schemes, we implicitly assume that it exists. The reader may keep in mind the case of quasi-projective schemes, where existence is guaranteed.
\end{notation}

\begin{example}
\label{ee}
Let $f:\Spec(L) \to \Spec(k)$ be the morphism corresponding to the Galois extension $L/k$ of degree $d$. Then Weil restriction sends affine space $\mathbb{A}^{n} _{L}$ to affine spaces $\mathbb{A}^{n\cdot d}_{k}$. In particular, it sends $\mathbb{A}^{0}_{L} \to \mathbb{A}^{0}_{k}$. This can be checked easily from the definition (See the proof of \cite[Theorem 4, \S 7.6]{MR1045822}).
But for projective spaces, the similar statement is not true in general. For example, we may consider the field extension $\R \hookrightarrow\mathbb{C}$. Then $\WR_{\mathbb{C}/\mathbb{R}}(\PP^{1}_{\mathbb{C}})$ is not isomorphic to $\PP^{1}_{\mathbb{R}}$.
\end{example}

We recall the main properties of the Weil restriction that will be used later.
The first proposition concerns the compatibility of Weil restriction with base change.
\begin{prop}\cite[Proposition 2.2]{MR4480537}
\label{p2}
    Let $f:S \to T$ be a finite étale morphism and $X$ be an $S$-scheme. Then we have
    \begin{itemize}
        \item[(1)] For a $T$-scheme $T_{i}$ and $T_{i}^{'}:=T_{i}\times_{T}S$, we have an isomorphism \[\WR_{T^{'}_{i}/T_{i}}(X\times_{S}T_{i} ^{'})\cong \WR_{S/T}(X)\times_{T} T_{i}.\]
        \item[(2)] Given morphisms of $S$-schemes $X\to Z$ and $Y \to Z$, we have an isomorphism \[\WR_{S/T}(X\times_{Z}Y)\cong \WR_{S/T}(X)\times_{\WR_{S/T}(Z)}\WR_{S/T}(Y).\]
    \end{itemize}
\end{prop}
The second proposition concerns properties of morphisms that are stable under Weil restriction.
\begin{prop}\cite[Proposition~2.6]{MR4480537}
\label{pp}
Let $f:S \to T$ be a finite étale morphism of schemes. For any morphism 
$i:X \to Y$ of $S$-schemes, we denote by 
\[
\WR(i):\WR(X)\to \WR(Y)
\]
the induced morphism. Then the following hold:
\begin{itemize}
    \item[(1)] If $i$ is an open embedding, then $\WR(i)$ is also an open embedding.
    \item[(2)] If $i$ is a closed embedding, then $\WR(i)$ is also a closed embedding.
    \item[(3)] If $i$ is smooth, then $\WR(i)$ is also smooth.
    \item[(4)] If $i$ is étale, then $\WR(i)$ is also étale.
\end{itemize}
\end{prop}
\subsection{Unstable motivic homotopy category and norm functors}
\label{se11}
In this subsection, we briefly recall the norm functors of Bachmann--Hoyois and relate it to the Weil restriction of schemes. For detailed discussion, we refer the reader to \cite{MR4288071}.

Let \(Z\) be a scheme, and recall that \(\textbf{Sm}_Z\) denotes the category of smooth schemes over \(Z\). The {\em unstable motivic homotopy category} over \(Z\), denoted by \(\mathcal H(Z)\), was introduced by Morel--Voevodsky \cite{MR1813224}. It is obtained from simplicial presheaves on \(\textbf{Sm}_Z\) by imposing Nisnevich descent and \(\mathbb A^1\)-homotopy invariance. An object of \(\mathcal H(Z)\) is called a {\em motivic space}. Any smooth \(Z\)-scheme \(X\) determines a {\em representable motivic space}, again denoted by \(X\), via the Yoneda embedding
\[
X(Y):=\Hom_Z(Y,X)
\]
for every \(Y\in \textbf{Sm}_Z\). 

By adjoining a disjoint base point, one obtains the {\em pointed motivic space}
\[
X_+:=X\sqcup * .
\]
This leads to the {\em pointed unstable motivic homotopy category} over \(Z\), denoted by \(\mathcal H_\bullet(Z)\). Equipped with the smash product, this category admits a symmetric monoidal structure.

Let $f:S\to T$ be any morphism of schemes. For an \(S\)-scheme \(X\), recall that its Weil restriction  \(\WR(X)\), if it exists, is characterized by the universal property \[\operatorname{Hom}_T(U,\WR(X))\cong
\operatorname{Hom}_S(U\times_T S,X)\]
for every \(T\)-scheme \(U\). 

However, Weil restriction of schemes by itself is not sufficient for the purposes of
motivic homotopy theory. We highlight three related issues.

\begin{itemize}
    \item[(1)] \textbf{Pointed motivic spaces.} Motivic homotopy theory is not
    built merely from schemes, but from pointed motivic spaces obtained after
    imposing Nisnevich descent and \(\mathbb A^1\)-homotopy invariance. Thus one
    needs a construction that acts functorially on pointed motivic spaces, not only
    on schemes.

    \item[(2)] \textbf{Monoidal structure.} Norm functors are inherently
    multiplicative. In particular, over a split finite étale cover, the expected
    operation is given by the smash product in pointed motivic spaces. This
    monoidal structure is not encoded by the scheme-theoretic Weil
    restriction alone.

    \item[(3)] \textbf{Stabilization.} Cohomology theories in motivic homotopy
    theory are represented by motivic spectra. Hence one needs compatibility with
    stabilization, so that the construction passes from the pointed unstable
    motivic homotopy category to the stable motivic homotopy category.
\end{itemize}

To address these issues, when $f:S \to T$ is finite étale, Bachmann--Hoyois \cite[\S 3, {\em Norms of pointed motivic spaces}]{MR4288071}
construct an {\em unstable norm functor}
\[
N_f=f_\otimes:\mathcal H_\bullet(S)\longrightarrow \mathcal H_\bullet(T)
\] between the pointed unstable motivic homotopy categories. This functor may be viewed as a pointed motivic refinement of the classical Weil restriction. More precisely, if \(X\) is a smooth \(S\)-scheme for which the Weil restriction \(\WR(X)\) exists, then there is a natural equivalence
\[
N_f(X_+)\simeq \WR(X)_+ .
\]
Here \(X_+\) denotes the pointed motivic space obtained from the representable motivic space \(X\); similarly for \(\WR(X)_+\).

Moreover, Bachmann--Hoyois prove that the norm functor \(N_f\) is symmetric monoidal and compatible with stabilization. Consequently, it descends to a norm functor, again denoted by
\[
N_f=f_\otimes:\mathcal{SH}(S)\longrightarrow \mathcal{SH}(T),
\]
between the stable motivic homotopy categories (or, more precisely, the motivic stable \(\infty\)-categories) of motivic spectra; see \cite[\S 2--\S 5]{MR4288071}. Under stabilization, we similarly obtain a natural equivalence
\[
N_f(\Sigma^\infty X_+)\simeq \Sigma^\infty \WR(X)_+ .
\]

\subsection{Norm functors and motives}
\label{propnorm}
In this subsection, we pass from the unstable motivic setting to Voevodsky's
category of motives. Bachmann--Hoyois construct norm functors at the level of
motivic \(\infty\)-categories, where the higher coherence required by the norm
formalism is encoded. After passing to homotopy categories, the full higher-coherent
structure is no longer visible, but the induced functors and natural equivalences
remain. Moreover, the \(\infty\)-categorical equivalences become natural isomorphisms
in the usual triangulated categories of motives.

The passage to motives is achieved via the adding-transfers functor constructed by Bachmann--Hoyois:
\[
Z^{tr}:\mathcal{SH}^{\otimes}\longrightarrow \mathcal DM^{\otimes}.
\]
They prove that this functor is compatible with norm functors; see
\cite[\S 14.1, ``Norms of presheaves with transfers'']{MR4288071}.
Here \(\mathcal DM^{\otimes}\) denotes the symmetric monoidal \(\infty\)-category corresponding to Voevodsky's motives. In particular, for a finite étale morphism $f:S\to T$,
one obtains a norm functor between Voevodsky's categories of motives
\[
N_f=f_\otimes:DM(S,\mathbb Z)\longrightarrow DM(T,\mathbb Z).
\]

Moreover, this construction admits a version with coefficients in a commutative ring \(R\); see \cite[Remark~14.6]{MR4288071}. Thus we also obtain a norm functor
\[
N_f=f_\otimes:DM(S,R)\longrightarrow DM(T,R).
\]

The compatibility of \(Z^{tr}\) with the unstable norm functor implies that, for a
smooth \(S\)-scheme \(X\) such that the Weil restriction \(\WR(X)\) exists, the
unstable formula
\[
N_f(X_+)\simeq \WR(X)_+
\]
induces the motivic formula
\[
N_f(M_S(X))\cong M_T(\WR(X)).
\]
Here \(M_S(X)\) and \(M_T(\WR(X))\) denote the motives of \(X\) over \(S\) and
of \(\WR(X)\) over \(T\), respectively.

We now record the properties of \(N_f\) that will be used in the following sections.

First, the norm functor is compatible with base change. Consider a Cartesian square
\[
\begin{CD}
S' @>{g'}>> S\\
@V{f'}VV @VV{f}V\\
T' @>{g}>> T,
\end{CD}
\]
where \(f\) is finite étale. Then the Bachmann--Hoyois norm formalism gives a
natural isomorphism
\[
g^*N_f(A)\cong N_{f'}(g'^*A)
\]
for \(A\in DM(S,R)\); see \cite[Proposition~5.3]{MR4288071}. This base-change
compatibility is one of the consequences of encoding norm functors by the
span-category formalism.

Second, for a split finite étale morphism, the norm is the monoidal product. More
precisely, if
\[
\nabla:\coprod_{i=1}^d T\to T
\]
is the fold map, then the norm functor along \(\nabla\) is given by the \(d\)-fold
tensor product in \(DM(T,R)\):
\[
N_\nabla(A_1,\ldots,A_d)\cong A_1\otimes\cdots\otimes A_d.
\]
At the level of pointed motivic spaces, this is the corresponding \(d\)-fold smash
product (\cite[Theorem 3.3]{MR4288071}).

\section{Weil restriction, vector bundles and normal bundles}
\subsection{Weil restriction of vector bundles}
In this section, we study the behavior of vector bundles under Weil restriction. 
From the perspective of motivic homotopy theory, a vector bundle determines its {\em Thom space}, which serves as a fundamental object in the theory. Consequently, understanding the behavior of vector bundles under Weil restriction is a natural first step toward comparing the associated Thom spaces and relating Weil restriction to norm functors.

We show that, under suitable assumptions, the Weil restriction of a vector bundle is again a vector bundle. This result is known in the context of unstable motivic homotopy theory and norm functors. To emphasize the underlying geometric nature of the construction, we include a more or less elementary proof. We begin with the following simple lemma.
\begin{lemma}
\label{aa}
Let $f:X \to Y$ be a finite étale morphism of schemes, and let $\mathcal{F}$ be a locally free sheaf of finite rank on $X$. Then $f_{*}(\mathcal{F})$ is locally free of finite rank on $Y$.
\end{lemma}

\begin{proof}
By \cite[Exercise~5.7, Chapter~II]{MR463157}, it suffices to check the freeness of stalks of $f_{*}\mathcal{F}$. By reducing to the case that $Y$ is an affine scheme given by a Noetherian local ring, the lemma follows from the standard fact that any finitely generated flat module over a noetherian local ring is free.\end{proof}

We are now ready to prove the following proposition.
\begin{prop}
\label{tttt}
    Let $f:S \to T$ be a finite étale morphism of schemes and $X$ a scheme over $S$. Suppose that $\WR(X)$ exists. Then for any vector bundle $E \to X$, $\WR(E)$ exists and $\WR(E)\to \WR(X)$ remains a vector bundle.
\end{prop}
\begin{proof}
    Let $\mathcal{E}$ be the associated sheaf of sections of $E$. Consider $\WR(X)_{S}:=\WR(X)\times  _{T}S$. Using the universal property of $\WR(X)$, we have a natural morphism $q:\WR(X)_{S}\to X$. Moreover, we have a finite étale morphism $p:\WR(X)_{S}\to \WR(X)$. Note that $q^{*}(\mathcal{E})$ is  a locally free sheaf  on $\WR(X)_{S}$. Applying Lemma \ref{aa}, we obtain that $p_{*}(q^{*}(\mathcal{E}))$ is a locally free sheaf on $\WR(X)$.

We set $\mathcal{J}:=p_{*}(q^{*}(\mathcal{E}))$ and claim that $\mathbf{Spec}(\Sym^{*}(\mathcal{J}^{\vee}))\cong \WR(E)$.

Indeed, Let $Z$ be an $\WR(X)$-scheme with structure morphism $u:Z\to \WR(X)$ and set $Z_S:=Z\times_T S$. By the universal property of relative spectrum, we have
\[
\Hom_{\WR(X)}\!\Bigl(Z,\mathbf{Spec}\bigl(\Sym^{*}(\mathcal J^\vee)\bigr)\Bigr)
\cong \Gamma(Z,u^*\mathcal J).
\]
By base change for the finite étale morphism $p:\WR(X)_{S} \to \WR(X)$, we have the following commutative diagram
\[
\begin{CD}
Z\times_{\WR(X)}\WR(X)_S @>{p_Z}>> Z \\
@VVV                      @V{u}VV \\
\WR(X)_S @>{p}>> \WR(X) \\
@V{q}VV \\
X
\end{CD}
\] We define $q_{Z}$ to be the composition of $Z\times_{\WR(X)}\WR(X)_S \to \WR(X)_{S}\to X$.
A finite morphism is affine, thus we have $u^*\mathcal J \cong p_{Z*}(q_Z^*\mathcal E)$. This shows that \[\Gamma(Z,u^*\mathcal J)\cong \Gamma(Z,p_{Z*}(q_Z^*\mathcal E))\cong \Gamma(Z_S,q_Z^*\mathcal E).\] Since $E$ is the vector bundle associated to $\mathcal E$, the latter group is identified with $\Hom_{X}(Z_S,E)$. 

On the other hand, by the universal property of $\WR(E)$,
\[
\Hom_{\WR(X)}(Z,\WR(E))\cong \Hom_X(Z_S,E).
\]
Thus
\[
\Hom_{\WR(X)}\!\Bigl(Z,\mathbf{Spec}\bigl(\Sym^{*}(\mathcal J^\vee)\bigr)\Bigr)
\cong
\Hom_{\WR(X)}(Z,\WR(E))
\]
functorially in $Z$, and the claim follows from Yoneda.
\end{proof}

\subsection{Weil restriction of normal bundles}
Normal bundles, as a special class of vector bundles, are closely related to regular embeddings in algebraic geometry. 
In the previous subsection, we showed that, under suitable assumptions, the Weil restriction of a vector bundle remains a vector bundle. 
A natural question is whether the Weil restriction of a normal bundle behaves compatibly with the corresponding closed embedding.

In fairly general situations, we give a positive answer to this question.
The result may be viewed as a generalization of a theorem originally due to Zhu and Karpenko \cite{MR4949893}. A careful inspection of the proof shows that the same argument applies in the present setting without essential modification.

Recall our convention that a scheme is separated and Noetherian. 
By Proposition~\ref{pp}, a closed embedding between smooth schemes (which is, in particular, a regular closed embedding) remains a closed embedding between smooth schemes (still, a regular closed embedding) under Weil restriction. 
Hence we have the following proposition.
\begin{prop}
\label{pppp1}
Let $f:S\to T$ be a \textbf{surjective} finite étale morphism of schemes and $i: Z \hookrightarrow X$ be a closed embedding of smooth $S$-schemes. We let $N_{Z/X}$ denote the associated normal bundle. Then there is an isomorphism
\[
\WR(N_{Z/X}) \cong N_{\WR(Z)/\WR(X)},
\]
where $N_{\WR(Z)/\WR(X)}$ denotes the normal bundle associated to the induced closed embedding
\[
\WR(i):\WR(Z)\hookrightarrow \WR(X)
\]
\end{prop}
\begin{proof}
We first prove the case when $S, T$ are affine schemes. Then the $S$-scheme $S\times_T S$
splits as a disjoint union $S\times_T S \cong S \sqcup S'$
for some finite étale $S$-scheme $S'$.

We set $\WR(X)_S := \WR(X)\times_T S$. By the standard properties of Weil restriction (cf.\ \cite[(4.2.3), (4.2.6)]{MR1321819}),
there is a decomposition
\[
\WR(X)_S \cong X \times R'(X_{S^{'}}),
\]
where $R'(X_{S^{'}})$ denotes the Weil restriction of $X_{S^{'}}$ along the finite étale morphism $f':S' \to S$.

Similarly, we have $\WR(Z)_S \cong Z \times R'(Z_{S^{'}})$.
Under these identifications, the morphism
\[
\WR(i)_S : \WR(Z)_S \to \WR(X)_S
\]
corresponds to
\[
i \times R'(i) : Z \times R'(Z_{S^{'}}) \to X \times R'(X_{S^{'}}),
\]
which is the product of regular closed embeddings.

By the compatibility of normal bundles with products \cite[Proposition~104.7]{EKM}, we obtain
\[
N_{\WR(Z)_S/\WR(X)_S}
\cong N_{Z/X} \times N_{R'(Z_{S^{'}})/R'(X_{S^{'}})}.
\]

Moreover, since the construction of the normal cone is compatible with flat base change (\cite[Proposition~104.23]{EKM}), we have
\[
N_{\WR(Z)_S/\WR(X)_S} \cong (N_{\WR(Z)/\WR(X)})\times_T S.
\]

Finally, there is a natural morphism
\[
N_{\WR(Z)/\WR(X)} \to \WR(N_{Z/X})
\]
which becomes an isomorphism after base change to $S$.
Since $S \to T$ is faithfully flat, we conclude by faithfully flat descent that
\[
N_{\WR(Z)/\WR(X)} \cong \WR(N_{Z/X}).
\]

We now treat the general case. Let $\{T_i\}_{i\in I}$ be an affine open cover of $T$, and set
\[
S_i:=S\times_T T_i,\qquad X_i:=X\times_S S_i,\qquad Z_i:=Z\times_S S_i.
\]
Since $f:S\to T$ is finite étale, each base change
\[
f_i:S_i\to T_i
\]
is again finite étale. Moreover, as $f_i$ is finite and $T_i$ is affine, the scheme $S_i$ is also affine. Since $f$ is surjective, each $f_i$ is surjective as well.

Applying the affine case to $f_i:S_i\to T_i$ and the closed immersion
\[
i_i:Z_i\hookrightarrow X_i,
\]
we obtain an isomorphism
\[
\WR_{f_i}(N_{Z_i/X_i}) \xrightarrow{\sim} N_{\WR_{f_i}(Z_i)/\WR_{f_i}(X_i)}.
\]

We now identify both sides with the restrictions of the corresponding global objects to $T_i$.

First, Weil restriction is compatible with base change (cf. Proposition \ref{p2}). Hence
\[
\WR_{f}(X)\times_T T_i \cong \WR_{f_i}(X_i),
\qquad
\WR_f(Z)\times_T T_i \cong \WR_{f_i}(Z_i),
\]
and similarly
\[
\WR_f(N_{Z/X})\times_T T_i \cong \WR_{f_i}(N_{Z_i/X_i}).
\]

Next, normal bundles are compatible with base change. So
we have
\[
N_{\WR_f(Z)/\WR_f(X)}\times_T T_i
\cong
N_{\WR_{f_i}(Z_i)/\WR_{f_i}(X_i)}.
\]

Therefore the affine-case isomorphism over $T_i$ may be rewritten as
\[
\bigl(\WR_f(N_{Z/X})\bigr)\big|_{T_i}
\xrightarrow{\sim}
\bigl(N_{\WR_f(Z)/\WR_f(X)}\bigr)\big|_{T_i}.
\]
By Proposition \ref{tttt}, both sides are vector bundles; it remains to check that these local isomorphisms agree on overlaps. Let
\[
T_{ij}:=T_i\cap T_j,
\qquad
S_{ij}:=S\times_T T_{ij}.
\]
Then $S_{ij}$ is canonically identified with both
\[
S_i\times_{T_i} T_{ij}
\quad\text{and}\quad
S_j\times_{T_j} T_{ij}.
\]
By construction, the isomorphisms over $T_i$ and $T_j$ are both obtained from the same affine-case argument after base change to $T_{ij}$. Since all constructions involved — Weil restriction, induced closed embeddings, and normal bundles — are functorial and compatible with base change, the two induced isomorphisms on $T_{ij}$ coincide.

Thus the family of local isomorphisms
\[
\bigl(\WR_f(N_{Z/X})\bigr)\big|_{T_i}
\xrightarrow{\sim}
\bigl(N_{\WR_f(Z)/\WR_f(X)}\bigr)\big|_{T_i}
\]
glues to a global isomorphism
\[
\WR_f(N_{Z/X}) \xrightarrow{\sim} N_{\WR_f(Z)/\WR_f(X)}.
\]
This proves the proposition.\end{proof}

\section{Motivic Thom spaces and Thom classes}
For a vector bundle, the associated Thom space gives rise to a distinguished cohomology class in motivic cohomology theory, called the \emph{Thom class}. In this section, we study the behavior of the Thom class under the Weil restriction map.

\subsection{Motivic Thom spaces and Weil restriction}
In this subsection, we show that the norm functor sends Thom spaces to Thom spaces.
This can be viewed as a direct result of Bachmann--Hoyois on the compatibility of norm functors
with quotients.

Let \(f:S\to T\) be a finite étale morphism, and let
\(\mathcal H_\bullet(S)\) and \(\mathcal H_\bullet(T)\) denote the pointed unstable
motivic homotopy categories over \(S\) and \(T\), respectively. Recall from
Section~\ref{se11} that Bachmann--Hoyois construct the norm functor
\[
N_f=f_\otimes:\mathcal H_\bullet(S)\longrightarrow \mathcal H_\bullet(T).
\]

Let \(E\to X\) be a vector bundle over \(X\in \mathbf{Sm}_S\), and let
\(s:X\hookrightarrow E\) denote the zero section. We define the Thom space of
\(E\) to be the pointed quotient
\[
\Th(E):=E/(E\setminus s(X))
\]
in \(\mathcal H_\bullet(S)\); see \cite[cf. \S 2, Spaces]{MR1648048}. The base
point is the image of the open complement \(E\setminus s(X)\).

We first record the following elementary compatibility of Weil restriction with
zero sections.

\begin{lemma}
\label{l2}
Let \(X\in \mathbf{Sm}_S\), and let \(E\to X\) be a vector bundle with zero
section \(s:X\hookrightarrow E\). Suppose that the Weil restriction of $X$ exists. Then
\[
\WR(s):\WR(X)\longrightarrow \WR(E)
\]
is the zero section of the vector bundle \(\WR(E)\to \WR(X)\).
\end{lemma}

\begin{proof}
Let \(\pi:E\to X\) be the bundle projection. Since \(s\) is a section of \(\pi\),
we have
\[
\pi\circ s=\id_X.
\]
Applying Weil restriction gives
\[
\WR(\pi)\circ \WR(s)=\WR(\id_X)=\id_{\WR(X)}.
\]
Thus \(\WR(s)\) is a section of the vector bundle \(\WR(E)\to \WR(X)\).

It remains to check that this section is the zero section. Let \(Y\) be a
\(T\)-scheme and let
\[
u:Y\to \WR(X)
\]
be a \(T\)-morphism. By the universal property of Weil restriction, \(u\)
corresponds to an \(S\)-morphism
\[
\widetilde u:Y\times_T S\to X.
\]
The composition
\[
\WR(s)\circ u:Y\to \WR(E)
\]
corresponds, again by the universal property of Weil restriction, to the morphism
\[
s\circ \widetilde u:Y\times_T S\to E.
\]
Since \(s\circ \widetilde u\) is the zero vector in the fiber of \(E\) over
\(\widetilde u\), it follows that \(\WR(s)\circ u\) is the zero vector in the fiber
of \(\WR(E)\) over \(u\). This holds functorially in \(Y\) and \(u\). Hence
\(\WR(s)\) is the zero section of \(\WR(E)\to \WR(X)\).
\end{proof}

We now prove the following compatibility. 

\begin{prop}
\label{t6}
Let \(f:S\to T\) be a finite étale morphism, let \(X\in \mathbf{Sm}_S\), and let
\(E\to X\) be a vector bundle with zero section \(s:X\hookrightarrow E\). Suppose that the Weil restriction of $X$ exists. Then \(R(E)\) exists by Proposition \ref{tttt} and there is a natural isomorphism in
\(\mathcal H_\bullet(T)\):
\[
N_f(\Th(E))\simeq \Th(\WR(E)).
\]
\end{prop}

\begin{proof}
Put \(Z:=s(X)\subset E\). By Bachmann--Hoyois
\cite[Proposition~3.13]{MR4288071}, applied to the closed immersion
\(Z\subset E\), we have
\[
N_f\bigl(E/(E\setminus Z)\bigr)
\simeq
\WR(E)/\bigl(\WR(E)\setminus \WR(Z)\bigr).
\]
By Lemma~\ref{l2}, the morphism
\[
\WR(s):\WR(X)\to \WR(E)
\]
is the zero section of the vector bundle \(\WR(E)\to \WR(X)\). Hence
\(\WR(Z)\) is the image of this zero section. Therefore
\[
\WR(E)/\bigl(\WR(E)\setminus \WR(Z)\bigr)
=
\Th(\WR(E)).
\]and the result follows.
\end{proof}

We will also use the corresponding statement in the setting of motives. Recall that the Thom space of a vector bundle is defined in Voevodsky's category of motives in the same way. Let \(L/k\) be a finite Galois extension, and write
\[
f:\Spec L\to \Spec k
\]
for the corresponding finite étale morphism. As recalled in Section~\ref{propnorm}, the adding-transfers functor $Z^{tr}$ of Bachmann--Hoyois gives rise to a norm functor
\[
N_f=f_\otimes:DM(L,\mathbb Q)\longrightarrow DM(k,\mathbb Q).
\]

Applying the adding-transfers functor to the isomorphism of Proposition~\ref{t6}, and using the compatibility of the adding-transfers functor with norm functors, we obtain a corresponding isomorphism

\begin{equation}
\label{gye}
    N_f(\Th(E))\cong\Th(\WR(E))
\end{equation}
in \(DM(k,\mathbb Q)\), where, by abuse of notation, \(\Th(E)\) and \(\Th(\WR(E))\) denote the motives of the corresponding Thom spaces. This is the motivic form of the Thom space compatibility that will be used below.

We conclude this subsection with the following theorem, which will be used in the next subsection.

\begin{thm}[{\cite[Proposition 14.5]{MR2242284}}]
\label{gy}
Let $k$ be a field. We let \(i: Z \hookrightarrow X\) be a closed embedding in \(\mathbf{Sm}_k\), and let \(N_{Z/X}\) be the associated normal bundle. Denote by
\[
M_Z(X):=X/(X\setminus Z)
\]
the quotient object in \(DM(k,\Z)\). Then there is a canonical isomorphism (the Gysin isomorphism)
\[
\Th(N_{Z/X}) \cong M_Z(X).
\]
Moreover, these objects fit into the distinguished triangle (Gysin triangle)
\[
M(X\setminus Z)\longrightarrow M(X)\longrightarrow M_Z(X)\longrightarrow M(X\setminus Z)[1] \]
\end{thm}

\subsection{Thom classes }
We fix a base field $k$ and let $L/k$ be a finite Galois extension with Galois group $G:=\Gal(L/k)$. Then $f:\Spec(L)\to \Spec(k)$ is a surjective finite étale morphism of constant degree. Throughout this subsection, we will always take the Weil restriction and norm functor along $f$.

Let $X$ be a smooth $L$-scheme. The \emph{cohomology group} of $X$ of bidegree $(p,q)$, denoted by $H^{p,q}(X,\Z)$, is defined as the cohomology group computed in $DM(L,\Z)$. 

Now we define the Thom classes. Let $X$ be a smooth $L$-scheme and let $E \to X$ be a vector bundle of rank $n$. Let $s:X \hookrightarrow E$ denote the zero section. It induces a Gysin morphism $s_* : H^{*,*}(X,\Z) \longrightarrow H^{*+2n,*+n}_X(E,\Z)$,
where $H^{*,*}_X(E,\Z)$ denotes the cohomology of $E$ with support on $X$, represented by $M_X(E)$.

Let $1_X \in H^{0,0}(X,\Z)$ be the unit element. Its image $s_*(1_X) \in H^{2n,n}_X(E,\Z)$ defines, under the canonical isomorphism 
$H^{*,*}_X(E,\Z) \cong H^{*,*}(\Th(E),\Z)$ given by the Gysin isomorphism (Theorem~\ref{gy}), the \emph{Thom class}
\[
\xi_E := s_*(1_X) \in H^{2n,n}(\Th(E),\Z).
\]

We now restrict ourselves to motivic cohomology with rational coefficients and suppose that \([L:k]=d\). 

Let \(X\) be a smooth \(L\)-scheme and let \(E\to X\) be a vector bundle.
Suppose that \(\WR(X)\) exists. Recall that, in this case, 
the Weil restriction map on motivic cohomology groups has been constructed in even cohomological degrees; see \cite[Compatibility, \S 4.3]{WMC}.

We will also use the analogous construction for motivic cohomology with
supports. Let \(i:Z\hookrightarrow Y\) be a closed immersion of smooth
\(L\)-schemes, and suppose that \(\WR(Y)\) and \(\WR(Z)\) exist. Then, by the
same method as in \cite[Compatibility, \S 4.3]{WMC}, one defines a Weil restriction map
\[
\WR:
H_Z^{p,q}(Y,\mathbb Q)
\longrightarrow
H_{\WR(Z)}^{dp,dq}(\WR(Y),\mathbb Q)
\]
for \(p\) even, where \(d=[L:k]\). Indeed, after base change to \(L\), the
target identifies with the \(G\)-invariant part of
\[
H_{\prod_{\sigma\in G} Z^\sigma}^{dp,dq}
\left(\prod_{\sigma\in G}Y^\sigma,\mathbb Q\right),
\]
and the image of a class \(\alpha\in H_Z^{p,q}(Y,\mathbb Q)\) is defined as
the unique class whose base change is the external product
\[
\prod_{\sigma\in G}\alpha^\sigma .
\]
The required Galois descent statement for cohomology with supports follows
from the Gysin triangle (Theorem \ref{gy}) together with the Galois
descent result established in \cite[Galois Descent, \S 4.2]{WMC}.

More precisely, we let $Z\hookrightarrow X$ be a closed embedding of smooth schemes over $k$. Applying the pullback functor along $f:\Spec(L)\to \Spec(k)$ to the Gysin triangle (Theorem \ref{gy}) associated with $Z\hookrightarrow X$, we have the following diagram of distinguished triangles \[\begin{tikzcd}
M(X_{L}\setminus Z_{L}) \arrow[r]   & M(X_{L}) \arrow[r]       & M_{Z_{l}}(X_{L}) \arrow[r]   & {M(X_{L}\setminus Z_{L})[1]}   \\
M(X\setminus Z) \arrow[r] \arrow[u] & M(X) \arrow[r] \arrow[u] & M_{Z}(X) \arrow[r] \arrow[u] & {M(X\setminus Z)[1]} \arrow[u]
\end{tikzcd}\]

Now we pass to the cohomology groups. Since the cohomology groups have rational coefficients, we can apply the natural $G:=\Gal(L/k)$-action and obtain the following commutative diagram of exact sequences
\[
\begin{tikzcd}
\cdots \arrow[r] & {H^{p-1,q}(X_{L}\setminus Z_{L},\Q)}^{G} \arrow[r]             & {H^{p,q}_{Z_{L}}(X_{L},\Q)^{G}} \arrow[r]     & {H^{p,q}(X_{L},\Q)^{G}} \arrow[r]            & \cdots \\
\cdots \arrow[r] & {H^{p-1,q}(X\setminus Z, \Q)} \arrow[r] \arrow[u, "\cong"] & {H^{p,q}_{Z}(X)} \arrow[u, "f^{}*"] \arrow[r] & {H^{p,q}(X,\Q)} \arrow[r] \arrow[u, "\cong"] & \cdots
\end{tikzcd}
\]
Then the desired Galois descent result follows from the standard $5$-Lemma.

It is known that the motivic spectrum representing motivic cohomology is a
\emph{normed spectrum} in the sense of Bachmann--Hoyois
\cite[Example~7.12]{MR4288071}. Therefore, by
\cite[Proposition~7.17]{MR4288071}, the norm functor induces operations on
motivic cohomology groups.

In the setting of motivic cohomology with supports, let
\(i:Z\hookrightarrow Y\) be a closed immersion of smooth \(L\)-schemes. We write
\[
M_Z(Y):=Y/(Y\setminus Z).
\]
Then motivic cohomology with supports is identified with the motivic cohomology
of this quotient:
\[
H_Z^{p,q}(Y,\mathbb Q)=H^{p,q}(M_Z(Y),\mathbb Q).
\]
Using the identification
\[
N_f(M_Z(Y))\simeq M_{\WR(Z)}(\WR(Y)),
\]
the normed structure induces an operation
\[
N_f^H:
H_Z^{p,q}(Y,\mathbb Q)
\longrightarrow
H_{\WR(Z)}^{dp,dq}(\WR(Y),\mathbb Q).
\]

In particular, for a vector bundle \(E\to X\) with zero section
\(s:X\hookrightarrow E\), taking \(Z=X\) and \(Y=E\) gives
\[
N_f^H:
H_X^{p,q}(E,\mathbb Q)
\longrightarrow
H_{\WR(X)}^{dp,dq}(\WR(E),\mathbb Q).
\]
Equivalently, using the Thom space identifications
\[
H_X^{p,q}(E,\mathbb Q)\simeq H^{p,q}(\Th(E),\mathbb Q),
\qquad
H_{\WR(X)}^{dp,dq}(\WR(E),\mathbb Q)
\simeq H^{dp,dq}(\Th(\WR(E)),\mathbb Q),
\]we have \[N_f^H:
H^{p,q}(\Th(E),\mathbb Q)
\longrightarrow
H^{dp,dq}(\Th(\WR(E)),\mathbb Q).\]
This is the norm operation on the corresponding Thom spaces, where we use
Equation~\ref{gye}.

The following lemma identifies this norm-induced operation with the Weil restriction map. This comparison will be used to prove that the Weil restriction operation sends Thom classes of a vector bundle to the Thom classes of its Weil restriction.

\begin{lemma}
\label{ab}
Let \(L/k\) be a finite Galois extension of degree \(d\), with Galois group
\(G\), and let \(f:\operatorname{Spec}L\to \operatorname{Spec}k\) be the
corresponding morphism. Let $i:Z\hookrightarrow Y$ be a closed immersion of smooth \(L\)-schemes such that \(\WR(Y)\) and
\(\WR(Z)\) exist. Under the natural identification
\[
N_f(M_Z(Y))\cong M_{\WR(Z)}(\WR(Y)),
\]
the norm-induced map on motivic cohomology with supports
\[
N_f^H:
H_Z^{p,q}(Y,\mathbb Q)
\longrightarrow
H_{\WR(Z)}^{dp,dq}(\WR(Y),\mathbb Q)
\]
agrees, for \(p\) even, with the Weil restriction map
\[
\WR:
H_Z^{p,q}(Y,\mathbb Q)
\longrightarrow
H_{\WR(Z)}^{dp,dq}(\WR(Y),\mathbb Q).
\]
\end{lemma}

\begin{proof}
Let \(G=\operatorname{Gal}(L/k)\). Consider the Cartesian square
\[
\begin{CD}
\operatorname{Spec}(L\otimes_k L) @>{g'}>> \operatorname{Spec}L\\
@V{f_L}VV @VV{f}V\\
\operatorname{Spec}L @>{g}>> \operatorname{Spec}k .
\end{CD}
\]
Since \(L/k\) is Galois, there is an isomorphism of \(L\)-algebras
\[
L\otimes_k L\simeq \prod_{\sigma\in G}L.
\]
Hence \(f_L\) is identified with the fold map
\[
\coprod_{\sigma\in G}\operatorname{Spec}L
\longrightarrow
\operatorname{Spec}L.
\]

By the base-change compatibility of the Bachmann--Hoyois norm functor and
the fact that the norm along a fold map is given by tensor products
(see Section \ref{propnorm}), we have
\[
g^*N_f(M_Z(Y))
\cong
N_{f_L}(g'^*M_Z(Y))
\cong
\bigotimes_{\sigma\in G}M_{Z^\sigma}(Y^\sigma).
\]
Moreover, the tensor product of motives with support identifies with the
motive with support of the product pair:
\[
\bigotimes_{\sigma\in G}M_{Z^\sigma}(Y^\sigma)
\cong
M_{\prod_{\sigma\in G}Z^\sigma}
\left(\prod_{\sigma\in G}Y^\sigma\right).
\]
On the other hand, by the base-change property of Weil restriction,
\[
\WR(Y)_L\simeq \prod_{\sigma\in G}Y^\sigma,
\qquad
\WR(Z)_L\simeq \prod_{\sigma\in G}Z^\sigma.
\]
Therefore
\[
M_{\WR(Z)}(\WR(Y))_L
\cong
M_{\prod_{\sigma\in G}Z^\sigma}
\left(\prod_{\sigma\in G}Y^\sigma\right).
\]

Let $\alpha\in H_Z^{p,q}(Y,\mathbb Q)$,
with \(p\) even. By the Galois descent isomorphism for motivic cohomology with supports, pullback to \(L\) identifies
\[
H_{\WR(Z)}^{dp,dq}(\WR(Y),\mathbb Q)
\]
with the \(G\)-invariant part of
\[
H_{\prod_{\sigma\in G}Z^\sigma}^{dp,dq}
\left(\prod_{\sigma\in G}Y^\sigma,\mathbb Q\right).
\]
Thus it is enough to compare the two classes after base change to \(L\).

By base-change compatibility of norms, the pullback of \(N_f^H(\alpha)\) to
\(L\) is the norm along \(f_L\) of \(g'^*\alpha\). Under the splitting
\[
\operatorname{Spec}(L\otimes_k L)\simeq
\coprod_{\sigma\in G}\operatorname{Spec}L,
\]
the class \(g'^*\alpha\) is identified with the family of conjugate classes
\[
(\alpha^\sigma)_{\sigma\in G},
\qquad
\alpha^\sigma\in H_{Z^\sigma}^{p,q}(Y^\sigma,\mathbb Q).
\]
Since \(f_L\) is the fold map, its norm is the external product
\[
\prod_{\sigma\in G}\alpha^\sigma
\in
H_{\prod_{\sigma\in G}Z^\sigma}^{dp,dq}
\left(\prod_{\sigma\in G}Y^\sigma,\mathbb Q\right).
\]
Because \(p\) is even, this external product is \(G\)-invariant. By definition of the Weil restriction map on motivic cohomology with supports, the class \(\WR(\alpha)\) is the unique class whose base change to \(L\) is
\[
\prod_{\sigma\in G}\alpha^\sigma.
\]
Therefore
\[
N_f^H(\alpha)=\WR(\alpha).
\]
\end{proof}

Finally, we want to show that the Weil restriction map sends Thom classes to Thom classes. 
\begin{thm}
\label{t5}
   Let \(E\to X\) be a vector bundle of rank \(n\) over a smooth
\(L\)-scheme, and suppose that \(\WR(X)\) exists. Under the natural identification
\[
N_f(\operatorname{Th}(E))\cong \operatorname{Th}(\WR(E)),
\]
the Weil restriction map of cohomology groups with rational coefficients sends the Thom class of \(E\) to the
Thom class of \(\WR(E)\):
\[
\WR(\xi_E)=\xi_{\WR(E)}
\in H^{2nd,nd}(\operatorname{Th}(\WR(E)),\Q).
\]
\end{thm}

\begin{proof}
    
We let $s:X \hookrightarrow E$ be the zero section embedding. By Lemma \ref{l2}, we have that $\WR(s):\WR(X)\hookrightarrow \WR(E)$ is the zero section embedding of the vector bundle $\WR(E)\to \WR(X)$. Suppose that $[L:k]=d$. We set $G:=\Gal(L/k)$. Using Proposition \ref{t6} in the setting of motives, we have 
$N_{f}(\Th(E))\cong \Th(\WR(E))$. Now we pass to the cohomology groups, which gives \[N_{f}^{H}: H^{2n,n}(\Th(E),\Q)\to H^{2dn,dn}(\Th(\WR(E)),\Q).\]

From the definition of Thom classes and Lemma \ref{ab}, it suffices to show that $\WR(s_{*}(1_{X}))=\WR(s)_{*}(1_{\WR(X)})$ viewed as an element in $H_{\WR(X)}^{2n\cdot d,n\cdot d}(\WR(E),\Q)$.

Note that the Weil restriction on cohomology preserves the unit. Indeed, for $1_X \in H^{0,0}(X)$, the class
$\prod_{\sigma\in G}(1_X)^\sigma$
is the unit on $\prod_{\sigma\in G} X^\sigma$, and hence its descended class is precisely $1_{R(X)}$. Here $X^\sigma$ denotes the scheme obtained by composing the structure morphism with the automorphism $\sigma:\Spec(L)\to \Spec(L)$ for $\sigma\in G$.

By construction, we have
\[
\bigl(\WR(s_*(1_X))\bigr)_L
=
\prod_{\sigma\in G} \bigl(s_*(1_X)\bigr)^\sigma,
\]
where $s_*^{\sigma}$ denotes the homomorphism on cohomology induced by the zero section embedding
$s^{\sigma}:X^{\sigma}\to E^{\sigma}$.

On the other hand, we have
\[
\bigl(\WR(s)_*(1_{\WR(X)})\bigr)_L
\cong
\prod_{\sigma\in G} s_*^{\sigma}(1_{X^{\sigma}}).
\]

By the base-change compatibility of Gysin morphisms for the Cartesian square obtained from \(s:X\hookrightarrow E\) by applying \(\sigma\), we have
\[
(s_*(1_X))^\sigma=s^\sigma_*(1_{X^\sigma}).
\] for every $\sigma\in G$. Therefore
\[
\prod_{\sigma\in G} \bigl(s_*(1_X)\bigr)^\sigma
=
\prod_{\sigma\in G} s_*^{\sigma}(1_{X^\sigma}).
\]

It follows that
\[
\bigl(\WR(s_*(1_X))\bigr)_L
=
\bigl(\WR(s)_*(1_{\WR(X)})\bigr)_L.
\]
Hence the theorem follows from the uniqueness of descent in our construction.\end{proof}


\bibliographystyle{acm}
\bibliography{Guangzhao}

\end{document}